\documentclass[letterpaper, 10 pt, conference]{ieeeconf}  

\IEEEoverridecommandlockouts                              

\title{\LARGE \bf
Bundle-based pruning in the max-plus curse of dimensionality free method
}

\author{Stephane Gaubert$^{1}$, Zheng Qu$^{2}$ and Srinivas Sridharan$^{3}$
\thanks{
}
\thanks{
}%
\thanks{
}%
}

\usepackage{stylefile}
\begin{document}

\maketitle
\thispagestyle{empty}
\pagestyle{empty}

\begin{abstract}

Recently a new class of techniques termed the max-plus curse of dimensionality-free methods have been developed to solve nonlinear optimal control problems. In these methods the discretization in state space is avoided by using a max-plus basis expansion of the value function. This requires  storing only the coefficients of the basis functions used for representation. However, the number of basis functions grows exponentially with respect to the  number of time steps of propagation to the time horizon of the control problem. This so called ``curse of complexity'' can be managed by applying a pruning procedure which selects the subset of basis functions that contribute most to the approximation of the value function. The pruning procedures described thus far in the literature rely on the solution of a sequence of high dimensional optimization problems which can become computationally expensive.

In  this paper we show that if the max-plus basis functions are linear and the region of interest in state space  is convex, the pruning problem can be efficiently solved by the bundle method. This approach combining the bundle method and semidefinite formulations  is applied to the quantum gate synthesis problem, in which the state space is the special unitary group (which is non-convex). This is based on the observation that the convexification of the unitary group leads to an exact relaxation. 
The results are studied and validated via examples.

\end{abstract}

\section{INTRODUCTION}

One general approach to the solution of optimal control problem is the dynamic programming principle, which 
in the deterministic case leads to a first-order, nonlinear partial differential equation, the Hamilton-Jacobi-Bellman Partial 
Differential Equation (HJB PDE). Classical numerical methods for solving the HJB PDE, as the finite difference scheme~\cite{CrandallLions83} or the semi-Lagrangian scheme~\cite{falcone,falcone-ferretti}, are all grid-based and known to suffer from the so called \firstdef{curse of dimensionality}, meaning that the 
number of grid points should grow exponentially with the space dimension.

In recent years, a new class of numerical methods has been developed after the work of Fleming and McEneaney~\cite{a5}, see~\cite{mceneaney-livre,a6,eneaneyphys,DowerMcE}.
These methods, named max-plus basis methods, exploit the linearity of the associated semigroup in the max-plus algebra. 
Only the time interval is discretized and at each discretized time step, the value function is approximated by a supremum (or infimum if the objective is minimized) of basis functions. Among several max-plus basis methods, the curse of dimensionality-free method
introduced by McEneaney~\cite{curseofdim} is of special interest because of its polynomial growth rate of the computational complexity with the space dimension. It
applies to the class of  optimal control problems where the Hamiltonian 
is given or approximated by a pointwise supremum of a finite number of ``simpler'' Hamiltonians. In particular, such Hamiltonians arise  when the control space
is discrete, for example in switched systems.
However, the number of basis functions is multiplied by the number of simpler Hamiltonians at each propagation.
Therefore in the practical implementation, a pruning operation removing at each propagation a certain number of basis functions less useful than others is required to attenuate this so called \firstdef{curse of complexity}.

In order to sort the basis functions for the pruning, an \firstdef{importance metric} is associated to each basis function. The latter measures the maximal lost caused by removing the corresponding basis function. The smaller the importance metric is, the less useful the basis function is. Hence, the pruning operation consists of sorting the basis functions by their importance metrics and selecting those with largest importance metrics. Hence the attenuation of the curse of complexity in the max-plus curse of dimensionality-free method is reduced to the calculus of importance metrics.

In the previous related works~\cite{a7,eneaneyphys,GaubertMQ11}, the importance metric is given or approximated as the optimal value of a convex semidefinite program and solved by the package CVX or YALMIP, calling the standard convex optimization solver SEDUMI or SDPT3. At each propagation in the max-plus curse of dimensionality-free method, if the number of basis functions is $m$, then we need to solve $m$ semidefinite programs of size $m$. The complexity of a standard convex optimization solver is polynomial to the program size $m$, in the worst case $O(m^{3.5})$~\cite{BoydVandenberghe}. However,
the number of basis functions $m$ is supposed to grow exponentially with respect to the number of propagation steps. Hence, it is necessary to develop a method more efficient than the general-purpose solver for the importance metric calculus when the number of basis functions $m$ is large.

Large-scale convex optimization has received many attentions. The most common approach in the literature is to reduce the complexity of the inversion of the linear equations in the interior point method, either by exploiting the sparsity~\cite[Chapter 3]{sra2012optimization} or by designing customized algorithms~\cite{LiuVandenberghe,KohKimBoyd,WallinHanssonJohansson}. 
A non-interior-point approach is the bundle method~\cite{HelmbergRendl}. In this paper, we first remark that when the basis functions are all linear and the state space is convex, the bundle method can be an alternative algorithm for the calculus of the importance metric. In order to apply the bundle method to the quantum gate synthesis problem~\cite{eneaneyphys} in which the state space is the unitary group, we need to convexify the unitary group. On the other hand, we show that the convexification of the unitary group leads to an exact relaxation. 
The efficiency of the bundle method is demonstrated via numerical examples, by comparison with the standard package CVX.

The paper is organized as follows. In Section~\ref{sec-1}, we review the general principle of the max-plus curse of dimensionality-free method, extract the importance metric calculus problem and propose Algorithm~\ref{alg:bundle} applying the bundle method.
In Section~\ref{sec-2}, we consider the quantum gate synthesis application. In Section~\ref{sec-3}, we show that the convexification of the unitary group leads to an exact relaxation. In Section~\ref{sec-4}, we present numerical results demonstrating the efficiency of the bundle method.

The notations used in the paper are the following. For  $a\in\cC$, $\Real(a)$ is the real part of $a$.
The space of $n\times n$ complex matrices is denoted by $M_n(\cC)$. For $X\in M_n(\cC)$, $X^*$ denotes its conjugate transpose. The space $M_n(\cC)$ is considered
as a real Hilbert space endowed with an inner product $\<\cdot,\cdot>$ given by:
$$
\<X_1,X_2>= \Real(\trace(X_1^*X_2)),\enspace\forall X_1,X_2\in M_n(\cC)\enspace.
$$
The induced norm of a matrix $X\in M_n(\cC)$ is denoted by $\|X\|$.
 The space of $n\times n$ positive semidefinite (resp. positive definite) matrices
is denoted by $\sym_n^+$ (resp. $\hat \sym_n^+$). For two Hermitian matrices $A,B\in M_n(\cC)$, we write $A\succeq B$ (resp. $A\succ B$) if $A-B \in \sym_n^+$ is
(resp. $A-B \in \hat\sym_n^+$).
 For $k\in \bN$,
we denote by $I_k$ the identity matrix of size $k$.  We denote by $U(n)$ the group of $n\times n$ unitary matrices and $B(n)$ the set of $n\times n$ matrices of spectral norm
less than 1:
$$
B(n):=\{X\in M_n(\cC):XX^*\preceq I_n\}\enspace.
$$

\section{MAX-PLUS CURSE OF DIMENSIONALITY FREE METHOD}\label{sec-1}
The main objective of this section is to introduce the pruning problem arising in the max-plus curse of dimensionality free method.
For this purpose, we first review briefly the  principle of the method in a general framework.

\subsection{Problem class}
Denote by $\cX\subset \R^d$ the state space and by $\cU\subset \R^m$  the control space. Let $x\in \cX$ and 
$T\in (0,+\infty]$.
Consider the following optimal control problem:
\begin{equation*}
 V_T(x):=\inf_{\bu} \int_0^T \ell(\bx(s),\bu(s)) ds +\phi (\bx(T)) \enspace,
\end{equation*}
where the state trajectory $\bx(\cdot):[0,T)\rightarrow \cX$ satisfies the dynamics:
\begin{equation*}
\dot \bx(s)=f(\bx(s),\bu(s)),\quad \bx(0)=x \in \cX
\enspace. 
\end{equation*}
The functions $f:\cX\times \cU\to\R^d$, $\ell:\cX\times \cU \to \R$  and $\phi:\cX\to\R$ represent  the dynamics, the running cost and the terminal cost, respectively. 
The value $V_{\cdot}(\cdot):\cX\times (0,T) \rightarrow \R$ gives the optimum of the objective as a function of the initial state $x$ and of the horizon $T$, called \firstdef{value function}.
In this general framework, we omit the necessary assumptions on the functions $f$, $\ell$ and $\phi$ to guarantee the existence and regularity of 
the value function.
The Hamiltonian associated to the above optimal control problem is:
$$
H(x,p)=\sup_{u\in \cU} \<p,f(x,u)>+\ell(x,u),\enspace \forall x,p\in \R^d\enspace.
$$
The corresponding Hamilton-Jacobi Partial Differential Equation (HJ PDE) is then:
\begin{equation}\label{eq-HJPDE}\left\{
\begin{array}{l}
 \frac{\partial V}{\partial t}(t,x)-H(x,\frac{\partial V}{\partial x}(t,x))=0,\qquad \forall(x,t)\in \cX\times (0,T)\enspace,\\
 V_0(x)=\phi(x),\quad \forall x\in \cX \enspace .\end{array}\right.
\end{equation}
The  Lax-Oleinik semigroup $(S_t)_{t\geq0}$ associated to the Hamiltonian $H$ is 
the evolution semigroup of the corresponding HJ PDE~\eqref{eq-HJPDE}, i.e.,
$$
S_t[\phi]=V_t,\quad\forall t\in(0,T)\enspace.
$$

In max-plus basis methods~\cite{a5}, we choose a set of \firstdef{basis functions} 
$\cB=\{\omega_i: \R^d\rightarrow \R\}_{i\in J}$ and approximate the value function $V_T$ by the infimum of a finite number of basis functions.
More precisely, we need to determine a subset $I$ and approximate $V_T$ as follows:
$$
V_T(x)\simeq \inf_{i\in I} \omega_i(x),\enspace\forall x\in \cX\enspace.
$$

The max-plus curse of dimensionality free method~\cite{curseofdim} 
applies to the class of  optimal control problems where the Hamiltonian 
is given or approximated by the supremum of ``simpler'' Hamiltonians:
$$
H(x,p)\simeq\sup_{m\in \M} H^m(x,p),\enspace \forall x,p\in \R^d\enspace,
$$
where $\M$ is a finite index set.
The term ``simpler'' refers to the condition that  for all $\omega_i \in \cB$, $t\geq 0$ and  $m\in \M$, the computation cost of $S_t^m[\omega_i]$ is 
polynomial to the state dimension $d$.
Moreover, we require that $S_t^m[\omega_i]$ is in $\cB$.

\begin{example}
In the original development of the method~\cite{curseofdim}, McEneaney considered the switching linear quadratic
optimal control problem where the Hamiltonian is given by the supremum of quadratic forms:
$$
H(x,p)=\sup_{m\in\M} (A^ mx)'p+\frac{1}{2}x'D^ mx+\frac{1}{2}p'\Sigma^ m p,\enspace \forall x,p\in \R^d\enspace.
$$
The matrices $\{A^m, D^m, \Sigma^m:m\in \M\}$ are parameters of the switching system.
The set of basis functions $\cB$ is chosen to be a set of bounded quadratic functions. 
The optimal control problem corresponding to each $H^m$ is  a linear quadratic control problem. 
Then for each quadratic basis function $\omega_i \in \cB$, $S_t^m[\omega_i]$ is  still a quadratic function  for all $t\geq 0$ and  $m\in \M$. Besides, the computation requires only the resolution of a Riccati differential equation, thus of cost $O(d^3)$.
\end{example}
\begin{example}\label{ex-quantum}The method finds application in the study of quantum circuit complexity in quantum optimal gate synthesis~\cite{eneaneyphys}. The related optimal control problem 
 is to find a least path-length trajectory on the special unitary group $SU(n)$. 
The corresponding Hamiltonian $H$ is given by:
$$
H(U,p)= \sup_{|v|=1} \<p, -i\{\sum_{k=1}^M v_k H_k\}U>-\sqrt{v^TRv}  \enspace,
$$
for all $p,U \in M_n(\cC)$. Here $\{H_1,\dots,H_M\} \subset M_n(\cC)$ are a generator of the Lie algebra of the special
unitary group $SU(n)$. The diagonal, symmetric and positive definite matrix $R\in M_n(\cC)$ is the weight matrix in the running cost.   Let $\cM=\{0,e_1,\dots,e_M\}$ denote the set of the zero vector and  the standard basis vectors
in $\R^M$. The authors proposed to approximate $H$  by:
$$
H(U,p)\simeq  \sup_{m\in \M} \<p, -i H_m U>-\sqrt{e_m^TRe_m} \enspace,
$$
for all $p,U \in M_n(\cC)$. The set of basis functions $\cB$ are chosen to be affine functions with linear part given by
a unitary matrix:
$$
\cB:=\{\omega(\cdot)=\<P,\cdot>+c:\enspace P\in U(n),c\in \R\}\enspace.
$$
The affine structure of the basis function is preserved by each semigroup $\{S_t^m\}_{t\geq 0}$  and the computation requires only a matrix multiplication, thus of cost $O(n^2)$.
\end{example}
\subsection{Principle of the method}

We discretize the time interval $[0,T]$ by small time step $\tau$.
The main idea of the max-plus curse of dimensionality free method is to approximate the semigroup $S_\tau$
by easily computable $\tilde S_\tau$:
$$
S_\tau\simeq \tilde S_\tau:=\inf_{m\in \M} S_\tau^m\enspace.
$$
 Let $N\in\bN$ such that $T=N\tau$.
First we approximate the value function $V_0=\phi$ by the infimum of a finite number of basis functions:
$$
V_0(x)\simeq \inf_{i\in I_0} \omega_i(x),\enspace\forall x\in \cX\enspace.
$$
Then we iterate for $k=1,\dots,N$:
\begin{equation}\label{a-maxite}
\begin{array}{ll}
V_{k\tau}&\simeq S_\tau[ \displaystyle\inf_{i\in I_{k-1}} \omega_i]\\
&\simeq \tilde S_\tau[ \displaystyle\inf_{i\in I_{k-1}} \omega_i]\\
&= \displaystyle\inf_{m\in \M} S_{\tau}^m[ \inf_{i\in I_{k-1}} \omega_i]\\
&=\displaystyle\inf_{m\in \M}\inf_{i\in I_{k-1}} S_{\tau}^m[\omega_i]\enspace,
\end{array}
\end{equation}
where the last equality follows from the max-plus linearity of the semigroup $\{S_{t}^m\}_{t\geq 0}$, see~\cite{a5}.

From the iteration equation~\eqref{a-maxite}, it is immediate that the number of basis functions is multiplied
by $|\M|$ at each iteration. If the computing each $S_{\tau}^m[\omega_i]$ requires $O(d^{\alpha})$ cost, then the total computation cost at the end
of $N$ iterations is $O(|\M|^N d^{\alpha})$. The most appealing characteristic of the method lies in its polynomial
growth rate in the state space dimension, compared to classical grid based methods.
 In this sense it is considered as a curse of dimensionality free method. However, in practical implementation of the method,  we need to incorporate a pruning operation, denoted by $\pP$, in order to 
reduce the number of basis functions:
\begin{equation}\label{a-maxitepr}
\begin{array}{ll}
V_{k\tau}&\simeq \displaystyle \pP\circ\tilde S_\tau[ \inf_{i\in I_{k-1}} \omega_i]\enspace\\
&\simeq \pP[\inf_{m\in \M}\inf_{i\in I_{k-1}} S_{\tau}^m[\omega_i]]\enspace.
\end{array}
\end{equation}

 Therefore, the pruning operation is a critical element in the practical implementation of the method, without which the number of basis functions explodes after a few iterations.

\subsection{Pruning techniques}

The pruning problem can be formulated as follows.
Let $\{\omega_0,\dots,\omega_m\}$ be a set of basis functions and
$$
\psi=\inf_{i=0,\dots,m} \omega_i\enspace.
$$
Then $\pP$ applied to $\psi$ approximates $\psi$ by selecting a subset $J\subset \{0,1,\dots,m\}$:
$$
\psi\simeq \pP[\psi]=\inf_{j\in J}\omega_j\enspace.
$$
The selection criteria can be the minimization of the approximation error for a limited cardinality
of $J$.
The latter problem was formulated in~\cite{GaubertMQ11} as a continuous $k$-median or $k$-center facility location problem, when minimizing
the $L_1$  or  $L_\infty$ approximation error. In most of the existing pruning algorithms, one basic task is to calculate the 
so called \firstdef{importance metric} of each basis function. The latter measures the maximal lost caused by removing the corresponding basis function.
More precisely, for each $j\in \{0,\dots,m\}$,
the importance metric $\delta_j$ of the basis function $\omega_j$ measured over the state space $\cX$, is defined by:
\begin{align}\label{a-importancemetric}
 \delta_j:=\max_{x\in \cX} \min_{ i\neq j} \omega_i(x)-\omega_j(x)\enspace.
\end{align}
In some cases, especially when the state space $X$ is not bounded, a normalization shall be considered, see for example~\cite{a7}.

Once we get all the values $\{\delta_0,\dots,\delta_m\}$, we can list them in non-increasing order and select the $k$ first corresponding
basis functions, as in~\cite{a7}. Or, we can efficiently generate witness points from the optimal solution, construct a $k$-center problem and apply some polynomial combinatorial algorithms, as in~\cite{GaubertMQ11}. Moreover, it is worth mentioning that if $\delta_j<0$
then the basis function $\omega_j$ can be pruned without error.

\subsection{Bundle method}\label{subsec-bundle}
So far we have seen that one central issue in the max-plus basis method is the pruning operation, which 
reduces to the calculus of the importance metrics $\{\delta_0,\dots,\delta_m\}$. In this section, we suppose that all the basis functions in $\cB$ are affine and focus on the calculus of 
the importance metric of the basis function $\omega_0$:
\begin{equation}\label{a-delta0}
 \begin{array}{ll}\mathrm{maximize} &\lambda\enspace,\\
\mathrm{subject~ to:}& x\in \cX;\\
& \lambda \leq \omega_i(x)-\omega_0(x),\enspace i=1,\dots,m\enspace.
\end{array}
\end{equation}

As we mentioned, in the max-plus curse of dimensionality free method,  the number of basis functions $m$ grows geometrically with respect to the number of iterations $N$ and needs to be large for improved precision in the max-plus basis method.
 It is therefore necessary to know how to deal with problem~\eqref{a-delta0} with large $m$.

If $\cX$ is convex,  then~\eqref{a-delta0} is a convex optimization problem (maximizing a concave function) and can be solved in polynomial time by the interior point based methods.
It is known that the interior point based methods have quadratic convergence: the number of iterations to yield the duality gap accuracy $\epsilon$ is $O(\sqrt{m}\ln(1/\epsilon))$~\cite{BoydVandenberghe}. However, each iteration one needs to solve a set of linear equations of size $O(m)$, called Newton equations.  Efficiency of the interior point based method depends on the complexity of the linear equations. General-purpose convex optimization packages like CVX~\cite{GrantBoydcvx} or YALMIP~\cite{yalmip} rely on sparse matrix factorizations to solve the Newton equations efficiently. While this approach is very successful in linear programming, it appears to be less effective for other classes of problems (for example, semidefinite programming)~\cite[Chapter 3]{sra2012optimization}. Some scalable customized interior point algorithms have been developed for large-scale convex optimization problems with non-sparse problem structure, for specific problem families~\cite{LiuVandenberghe,KohKimBoyd,WallinHanssonJohansson}.

The main purpose of this paper is to propose a general and scalable  algorithm for solving~\eqref{a-delta0}, with no other requirement on the problem structure than the linearity of basis functions in $\cB$ and the convexity of $\cX$.
 Our approach is the bundle method, known for solving large-scale non-smooth convex optimization problems~\cite{Kiwiel90,SchrammZowe,LemarechalNemirovskiNesterov}.  We mention that the bundle method has already been exploited as an alternative of the interior-point method for large-scale semidefinite programming~\cite{HelmbergRendl} and shows considerably improved efficiency.

Denote:
$$
\phi(x)=\min_{i=1,\dots,m} \omega_i(x)-\omega_0(x),\enspace \forall x\in \cX\enspace.
$$
The basic principle of the bundle method is to use limited number of supporting affine functions to approximate the objective function $\phi$.
 The algorithm can be described as follows.



\begin{algorithm}[h]
\begin{algorithmic}[1]
\STATE{$\mathbf{Parameter}$: $\mu>0$; $\epsilon>0$; $\gamma>0$}
\STATE{$\mathbf{Input}$:  an initial point $x_1\in \cX$; a set $J_0\subset \{1,\dots,m\}$;}
\FOR{$k=1,2,\dots$}
\STATE $i_k\leftarrow \argmin\{ \omega_i(x_k)-\omega_0(x_k): i=1,\dots,m\}$; 
\STATE $J_k=J_{k-1}\cup \{i_k\}$;
\STATE Update model $$\phi_k^{CP}:=\inf_{i\in J_k} {\omega_i-\omega_0}\enspace;$$
\STATE $v_k\leftarrow \phi(x_k)$;
\STATE $y_k\leftarrow \argmax \{\phi_k^{CP}(y): y\in \cX; |y-x_k|\leq \mu\}$;
\STATE $w_k \leftarrow \phi_k^{CP}(y_k)$;
\IF{ $w_k-v_k<\epsilon$ and $|y_k-x_k|<\mu$}
\STATE stop;
\ELSE
\IF{$\phi(y_k)-v_k>\gamma (w_k-v_k)$}
\STATE $x_{k+1}\leftarrow y_k$;
\ELSE
\STATE $j\leftarrow \argmin\{\omega_i(y_k)-\omega_0(y_k): i=1,\dots,m \}$ ;
\STATE $J_k=J_k\cup \{j\}$;
\ENDIF
\ENDIF
\ENDFOR
\STATE{$\mathbf{Output}$: optimal value $v_k$; optimal solution $x_k$.}
\end{algorithmic}
\caption{Trust region Bundle method}
\label{alg:bundle}
\end{algorithm}

At iteration $k$, we approximate the objective function $\phi$ by $\phi_k^{CP}$,
which is the infimum of the supporting hyperplanes of the sequence of points
$\{y_1,\dots,y_{k-1}\}$.
 The next point $y_k$ is the maximizer of $\phi_k^{CP}$ on a region close to the current center $x_k$. It is the solution of the following optimization 
problem:
\begin{equation}\label{a-subprob}
 \begin{array}{ll}\mathrm{maximize} &\lambda\enspace,\\
\mathrm{subject~ to:}& y\in \cX;\\
& \lambda \leq \omega_i(y)-\omega_0(y),\enspace i\in J_k\enspace;\\
& |y-x_k|\leq \mu\enspace.
\end{array}
\end{equation}
Apart from the proximal constraint, this optimization problem
is of the same form as~\eqref{a-delta0} but with only $|J_k|$ linear constraints. Besides, since we add at most one constraint per iteration, we know that $|J_k|\leq k$.

\begin{rem}\label{rem-cb}
For a given convex state space $\cX$ and a given set of linear basis functions $\cB$, 
denote by $c(m)$ the maximal computation cost required by a standard convex optimization solver for solving~\eqref{a-delta0}. 
Then the complexity of Algorithm~\ref{alg:bundle} is bounded by:
$$
\sum_{k=1}^{K_0} (c(k)+O(m))\enspace.
$$
where $K_0$ denotes the number of iterations of Algorithm~\ref{alg:bundle}. From the latter expression on the complexity bound we read out the central thought of the bundle method: 
solve a large-scale convex optimization problem by solving a sequence of smaller size convex optimization problem. 
\end{rem}

\section{Application to the quantum optimal synthesis example}\label{sec-2}

In this section, we extract the pruning problem appearing in the quantum optimal synthesis application, presented
in Example~\ref{ex-quantum}. For more background we refer to~\cite{eneaneyphys}.

As we mentioned in Example~\ref{ex-quantum}, the state space $\cX$ is the special unitary group $SU(n)$ and 
the basis functions  are chosen to be affine functions:
$$
\omega_i(X)=\<P_i,X>+c_i,\enspace\forall X\in M_n(\cC)\enspace,\enspace i=0,\dots,m\enspace,
$$
for some $\{P_0,P_1,\dots,P_m\}\subset U(n)$  unitary matrices and $\{c_0,c_1,\dots,c_m\}\subset \R$.
Then the optimization problem in the pruning procedure can be described as:
\begin{align}\label{a-optiSU}
\delta_0=\max_{X\in SU(n)} \min_{1\leq i\leq m} \<P_i-P_0, X>+c_i-c_0\enspace.
\end{align}
We release the constraint on the determinant and compute the following upper bound of $\delta_0$:
\begin{align}\label{a-optiU}
\bar \delta_0=\max_{X\in U(n)} \min_{1\leq i\leq m} \<P_i-P_0, X>+c_i-c_0\enspace.
\end{align}
In order to obtain a convex optimization problem, we consider
a relaxation of~\eqref{a-optiU}:
\begin{align}\label{a-relaxU}
 \max_{X\in B(n)} \min_{1\leq i\leq m} \<P_i-P_0, X>+c_i-c_0\enspace.
\end{align}
By Schur's complement lemma~\cite[Lemma 6.3.4]{BenTalGhaouiNemirovski}, the constraint $X \in B(n)$ is equivalent to the following 
semidefinite matrix constraint:
$$
\left( \begin{array}{ll} I_n  & X \\ X^* & I_n\end{array}\right)\succeq 0\enspace.
$$
Hence problem~\eqref{a-relaxU} falls into the class of \firstdef{disciplined} convex programming problems~\cite{GrantBoydYinyu} and can be solved by calling directly the package CVX, or by the bundle method described in Algorithm~\ref{alg:bundle}.

\section{CONVEXIFICATION OF THE UNITARY GROUP}\label{sec-3}
In this section
we show that~\eqref{a-relaxU} is an exact relaxation of~\eqref{a-optiU}. Our main result is Theorem~\ref{th-main}. We first prove some useful lemmas.
\begin{lemma}\label{l-interiorimpos}
 Let $\{P_i\}_{i=0,\dots,m}$ be a set of $n\times n$ unitary matrices and $X\in M_n(\cC)$. If $XP_0$ is an interior point in the convex hull of
$\{XP_i\}_{i=1,\dots,m}$, then $XP_i=XP_0$ for all $i=1,\dots,m$. 
\end{lemma}
\begin{proof}
 It is immediate that if $XP_0$ is an interior point in the convex hull of
$\{XP_i\}_{i=1,\dots,m}$, then 
there are $\alpha_1,\dots,\alpha_m>0$  such that
$$
\sum_{i=1}^m\alpha_i=1,\enspace \sum_{i=1}^m \alpha_i XP_i =XP_0\enspace.
$$
The Frobenius norm of $XP_i$ equals to that of $X$ for all $i=0,\dots,m$. By the strict convexity of the Frobenius norm, we deduce that necessarily
$$
XP_i=XP_0\enspace,\quad i=1,\dots,m\enspace.
$$
\end{proof}

The \firstdef{tangent cone} to $B(n)$ at $X\in B(n)$ is defined by~\cite[p.204]{RockafellarWets}:
$$
T_{B(n)}(X)=\cl \{ \lambda(Z-X): \lambda\geq 0, Z \in B(n) \}\enspace.
$$ 
\begin{lemma}\label{l-tangentconeofdiagonal}
Let $1\leq k\leq n$ and $
\Sigma$ be a diagonal matrix with positive real diagonal entries $(\lambda_1,\cdots,\lambda_n)$ such that $\lambda_j=1$ for all $j=1,\dots,k-1$ and $\lambda_j<1$ for all $j=k,\cdots,n$. Then
\begin{align*}
\big\{ \left( \begin{array}{ll} X_1  & X_2 \\ X_3 & X_4\end{array}\right)\in M_n(\cC): -X_1-X_1^*\in \hat \sym_{k-1}^+  \big\}\subset T_{B(n)}(\Sigma)\enspace.
\end{align*}
\end{lemma}
\begin{proof}
If $k=1$ then it is clear that $\Sigma$ is an interior point of $B(n)$ and the tangent cone at $\Sigma$ is the whole space $M_n(\cC)$.

We next consider the case when $1<k\leq n$. For ease of proof, we write $\Sigma$  into block matrix:
$$
\Sigma=\left( \begin{array}{ll} I_{k-1}  & 0 \\ 0 & \Sigma_4\end{array}\right)\enspace,
$$
where  $\Sigma_4$ is a diagonal matrix such that $I_{n-k+1}-\Sigma_4\Sigma_4^* \succ 0$.
  Let
$$
X=\left( \begin{array}{ll} X_1  & X_2 \\ X_3 & X_4\end{array}\right)\in M_n(\cC)\enspace,
$$
such that $-X_1-X_1^*\in \hat \sym_{k-1}^+ $. For notational simplicity, let $H_1\in M_{k-1}(\cC)$, $H_4\in M_{n-k+1}(\cC)$ and
$H_2\in \cC^{(k-1)\times (n-k+1)}$ such that:
$$
XX^*=\left( \begin{array}{ll} H_1  & H_2 \\ H_2^* & H_4\end{array}\right)\enspace.
$$
Let any $\Delta >0$. We have:
\begin{align*}
&(\Sigma+\Delta X)(\Sigma+\Delta X)^*-I_n\\
&=\Sigma\Sigma^*+\Delta (X\Sigma^*+ \Sigma X^*)-I_n+\Delta^2 XX^*\\
&=\left( \begin{array}{ll} \Delta (X_1+X_1^*)  & \Delta (X_2\Sigma_4^*+X_3^*) \\ \Delta(\Sigma_4X_2^*+X_3)  & \Sigma_4\Sigma_4^*-I_{n-k+1} +\Delta(\Sigma_4X_4^*+X_4\Sigma_4^*)\end{array}\right)\\
&\qquad+\left( \begin{array}{ll} \Delta^2 H_1  & \Delta^2 H_2 \\ \Delta^2 H_2^*  & \Delta^2 H_4\end{array}\right)
\end{align*}
By Schur's complement lemma~\cite[Lemma 6.3.4]{BenTalGhaouiNemirovski}, we know that $(\Sigma+\Delta X)(\Sigma+\Delta X)^*\prec I_n$ if and only if $X_1+X_1^*+\Delta H_1\prec 0$ and
\begin{align*}
\begin{array}{l}
I_{n-k+1}-\Sigma_4\Sigma_4^*-\Delta (\Sigma_4X_4^*+X_4\Sigma_4^*)-\Delta^2 H_4\\+\Delta  (\Sigma_4X_2^*+X_3+\Delta H_2^*)(X_1+X_1^*+\Delta H_1)^{-1}(X_2\Sigma_4^*+X_3^*+\Delta H_2)
\\ \succ 0\enspace.
\end{array}
\end{align*}
Since $X_1+X_1^*\prec 0$ and $I_{n-k+1}-\Sigma_4\Sigma_4^* \succ 0$, there is $\Delta>0$ such that the above inequalities hold thus $(\Sigma+\Delta X)(\Sigma+\Delta X)^*\prec I_n$.
Hence, $X$ is in the tangent cone of $B(n)$. 
\end{proof}

In the sequel, let $\{P_i\}_{i=0,\dots,m}$ be a set of $n\times n$ unitary matrices and $\{c_i\}_{i=0,\dots,m}$ be a set of real numbers. 
For all $\beta>0$, denote:
\begin{align}\label{a-phibeta}
\phi_{\beta}(X)=-\beta^{-1} \log(\sum_{i=1}^m e^{-\beta(\<P_i-P_0, X>+c_i-c_0)}),\enspace \forall X\in M_n(\cC)\enspace.
\end{align}
\begin{lemma}\label{l-equicontphi}
 There is $K>0$ such that $\phi_\beta$ is $K$-Lipschitz continuous for all $\beta>0$.
\end{lemma}
\begin{proof}
Let $\beta>0$. For all $X,Y\in M_n(\cC)$ we have:
\begin{align*}
D\phi_{\beta}(X) \circ Y
=\sum_{i=1}^m \alpha_i \<P_i-P_0,  Y >\enspace,
\end{align*}
where 
$$
\alpha_i=(\sum_{i=1}^m e^{-\beta(\<P_i-P_0, X>+c_i-c_0)})^{-1}e^{-\beta(\<P_i-P_0, X>+c_i-c_0)} \enspace.
$$
Thus $\alpha_1,\dots, \alpha_m>0$ and $\sum_{i=1}^m \alpha_i=1$.
Hence, for all $X,Y\in M_n(\cC)$,
$$
\|D\phi_{\beta}(X) \circ Y\|\leq \sum_{i=1}^m \alpha_i \|P_i-P_0\|\|Y\|\leq \max_i \|P_i-P_0\| \|Y\|\enspace.
$$
It follows that for all $X\in M_n(\cC)$,
$$\|D\phi_{\beta}(X)\|\leq \max_i\|P_i-P_0\|\enspace.$$
Therefore let $K=\max_i\|P_i-P_0\|$, the function $\phi_\beta$ is $K$-Lipschitz for all $\beta>0$.
\end{proof}

\begin{prop}\label{p-minimizerofphibeta}
Let $\beta>0$.
The optimal solution of the following optimization problem
\begin{align}\label{a-PiPjcicj}
\max_{X\in B(n)} \phi_{\beta}(X)\enspace
\end{align} contains a unitary matrix.
\end{prop}
\begin{proof}
Let $U_0\in B(n)$ be an optimal solution of~\eqref{a-PiPjcicj}. Suppose that $U_0$ is not unitary.
 Consider the SVD decomposition of $U_0$ given by
 $$U_0=V_1 \Sigma V_2\enspace,$$ where $\Sigma $ is a
diagonal matrix with positive real diagonal entries $(\lambda_1,\dots,\lambda_n)$, listed in non-increasing order. Let $k\in\{1,\dots,n\}$ such that $\lambda_i=1$ for all $i=1,\dots,k-1$ and $\lambda_j<1$ for all
$j=k,\dots,n$. Then $\Sigma$ is an optimal solution of the following optimization problem:
\begin{align}\label{a-PiPjcicj2}
\max_{X\in B(n)} \phi_{\beta}( V_1XV_2)\enspace.
\end{align}
The first-order optimality condition~\cite[p.207]{RockafellarWets} implies that
\begin{align}\label{a-necoptcond}
D\phi_{\beta}(V_1\Sigma V_2) \circ (V_1 Y V_2) \leq 0,\quad \forall Y \in T_{B(n)}(\Sigma)\enspace.
\end{align}
We have:
\begin{align*}
D\phi_{\beta}(V_1\Sigma V_2) \circ(V_1YV_2)
=\sum_{i=1}^m \alpha_i \<P_i-P_0, V_1 Y V_2>\enspace,
\end{align*}
where $\alpha_1,\dots, \alpha_m>0$ and $\sum_{i=1}^m \alpha_i=1$.
Therefore,
$$
D\phi_{\beta}(V_1\Sigma V_2) \circ(V_1YV_2)=\<\sum_{i=1}^m \alpha_i(V_1^*P_iV_2^*-V_1^*P_0V_2^*), Y>
$$
 By the first-order optimality condition~\eqref{a-necoptcond} and Lemma~\ref{l-tangentconeofdiagonal}, we deduce that
$$ \<\sum_{i=1}^m \alpha_i(V_1^*P_iV_2^*-V_1^*P_0V_2^*), X>\leq 0\enspace,$$
for all $$
X=\left( \begin{array}{ll} X_1  & X_2 \\ X_3 & X_4\end{array}\right)\in M_n(\cC)\enspace,
$$
such that $-X_1-X_1^*\in \hat \sym_{k-1}^+ $. 
Hence,
$$
\sum_{i=1}^m \alpha_i(V_1^*P_iV_2^*-V_1^*P_0V_2^*)
=\left( \begin{array}{ll} Z  & 0 \\ 0 & 0\end{array}\right)
$$
for some $Z$ such that $Z+Z^*\in \sym_{k-1}^+$.
Therefore,
$$
(I_n-\Sigma)\sum_{i=1}^m \alpha_i (V_1^*P_iV_2^*-V_1^*P_0V_2^*)=0\enspace.
$$
By Lemma~\ref{l-interiorimpos}, we know that 
$$
(I_n-\Sigma)V_1^*P_iV_2^* =(I_n-\Sigma)V_1^*P_0V_2^*, \enspace i=1,\dots,m\enspace.
$$
This implies that
$$
\<P_i-P_0,V_1V_2>=\<P_i-P_0, V_1 \Sigma V_2>,\enspace i=1,\dots,m\enspace.
$$
Therefore,
$$
\phi_{\beta}(V_1 V_2)=\phi_{\beta} (V_1 \Sigma V_2)=\max_{X\in B(n)} \phi_{\beta} (X )\enspace.
$$
Hence $V_1V_2$ is an optimal solution of~\eqref{a-PiPjcicj}.
\end{proof}

\begin{theo}\label{th-main}
The set of optimal solutions of the following optimization problem:
\begin{align}\label{a-PiPjcicj3}
\max_{ X\in B(n)} \min_{1\leq i\leq m} \Real(\<P_i-P_0, X>)+c_i-c_0\enspace
\end{align}
 contains a unitary matrix.
\end{theo}
\begin{proof}
 Denote
$$
\phi(X)=\min_{1\leq i\leq m} \Real(\<P_i-P_0, X>)+c_i-c_0,\enspace\forall X\in M_n(\cC)\enspace.
$$
By Lemma~\ref{l-equicontphi} and the Arzel\`a-Ascoli theorem, the function $\phi_{\beta}$ defined in~\eqref{a-phibeta} converges uniformly to $\phi$ as $\beta$ goes to $+\infty$.
For each $\beta$, by Proposition~\ref{p-minimizerofphibeta} the intersection
 $$ O_\beta:=U(n) \cap \underset{X\in B(n)}{\argmax} \phi_{\beta}$$ is not empty.  Since the convergence of $\phi_{\beta}$ to $\phi$ is uniform, each cluster point of a sequence $\{U_{\beta}\}_{\beta\geq 0}$ with $U_{\beta} \in O_\beta$ for all $\beta>0$
 is an optimal solution of the problem~\eqref{a-PiPjcicj3}, see~\cite[p.266]{RockafellarWets}.
The cluster point is unitary because $U(n)$ is closed.  Thus the optimization problem~
\eqref{a-PiPjcicj3} must have a unitary optimal solution.
\end{proof}

By Theorem~\ref{th-main}, solving~\eqref{a-relaxU} is equivalent to solving~\eqref{a-optiU}. 
\section{Numerical examples}\label{sec-4}
 We implemented Algorithm~\ref{alg:bundle} for solving~\eqref{a-relaxU} in  Matlab (version 8.1.0.604 (R2013a)). The instances are generated during the propagation in the max-plus curse of dimensionality free 
method applied to  an optimal control problem on $SU(4)$ arising in the quantum optimal gate synthesis~\cite{eneaneyphys}. The parameters in Algorithm~\ref{alg:bundle} are chosen to be as follows: $\mu=0.5$, $\epsilon=1$e-8 and $\gamma=0.5$.
Every convex optimization problem within Algorithm~\ref{alg:bundle} is
solved by the standard package CVX~\cite{GrantBoydcvx} with solver SDPT3~\cite{sdpt4}.
To make a comparison, we also solved the same instances of~\eqref{a-relaxU} by the interior point algorithm (using the package CVX and calling the solver SDPT3).
The computations were performed on a single core of an Intel 12-core running
at 3GHz, with 48Gb of memory.

We compare in Figure~\ref{figurelabel} the computation time in seconds of solving one single instance of the problem~\eqref{a-relaxU}, for a different number of basis functions $m$, via two methods:  \roml{1}. the bundle method (described in Algorithm~\ref{alg:bundle})  and \roml{2}. the interior point method (using the package CVX).
 We observe that the time required by the interior point method (the green curve) grows much faster than  
the time required by  the bundle method (the blue curve). The stable computation cost of the bundle method 
with respect to the number of basis functions $m$ is, as we mentioned in Section~\ref{subsec-bundle}, critical for improving precision order in 
the max-plus curse of dimensionality free method.

In Figure~\ref{figurelabel2} we provide more details of the numerical results. We show in three subfigures the computation time in seconds and the number of iterations $K_0$ of the bundle method as well as 
the difference between the optimal value obtained by the bundle method $\delta_{bundle}$ and the one obtained by the interior point method $\delta_{cvx}$. 
Note that  for almost all tested $m$, the difference $\delta_{cvx}-\delta_{bundle}$ is less than 5e-8, which is of the same order as the duality gap obtained by
the package CVX.
It is interesting to remark that the number of iterations $K_0$ of the bundle method does not seem to increase as $m$ increases.
This explains why the computation time of the bundle method grows  slowly with respect to $m$: as we mentioned in Remark~\ref{rem-cb} the computation time required by 
Algorithm~\ref{alg:bundle} is bounded by
$$
\sum_{k=1}^{K_0} (c(k)+O(m))\enspace.
$$
Hence if $K_0$ does not increase with $m$, the computation time required by Algorithm~\ref{alg:bundle} is only linear to $m$.

   \begin{figure}[thpb]
      \centering
      \includegraphics[scale=0.4]{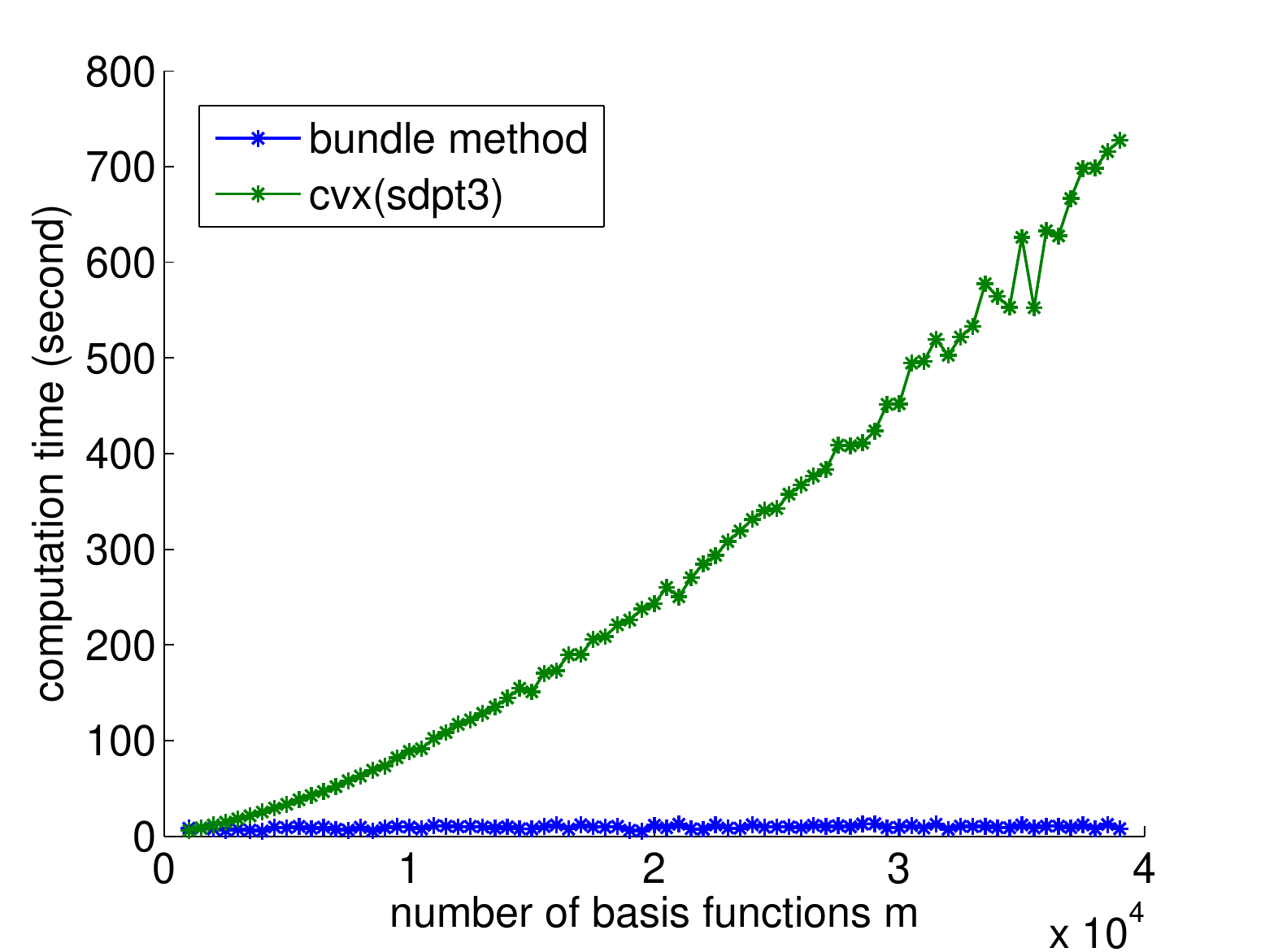}
      \caption{Computation time in seconds V.S. the number of basis functions $m$, obtained by the bundle method (blue curve) and by the interior point algorithm (green curve) for solving~\eqref{a-relaxU}}
      \label{figurelabel}
   \end{figure}
   
 \begin{figure}[thpb]
      \centering
      \includegraphics[scale=0.4]{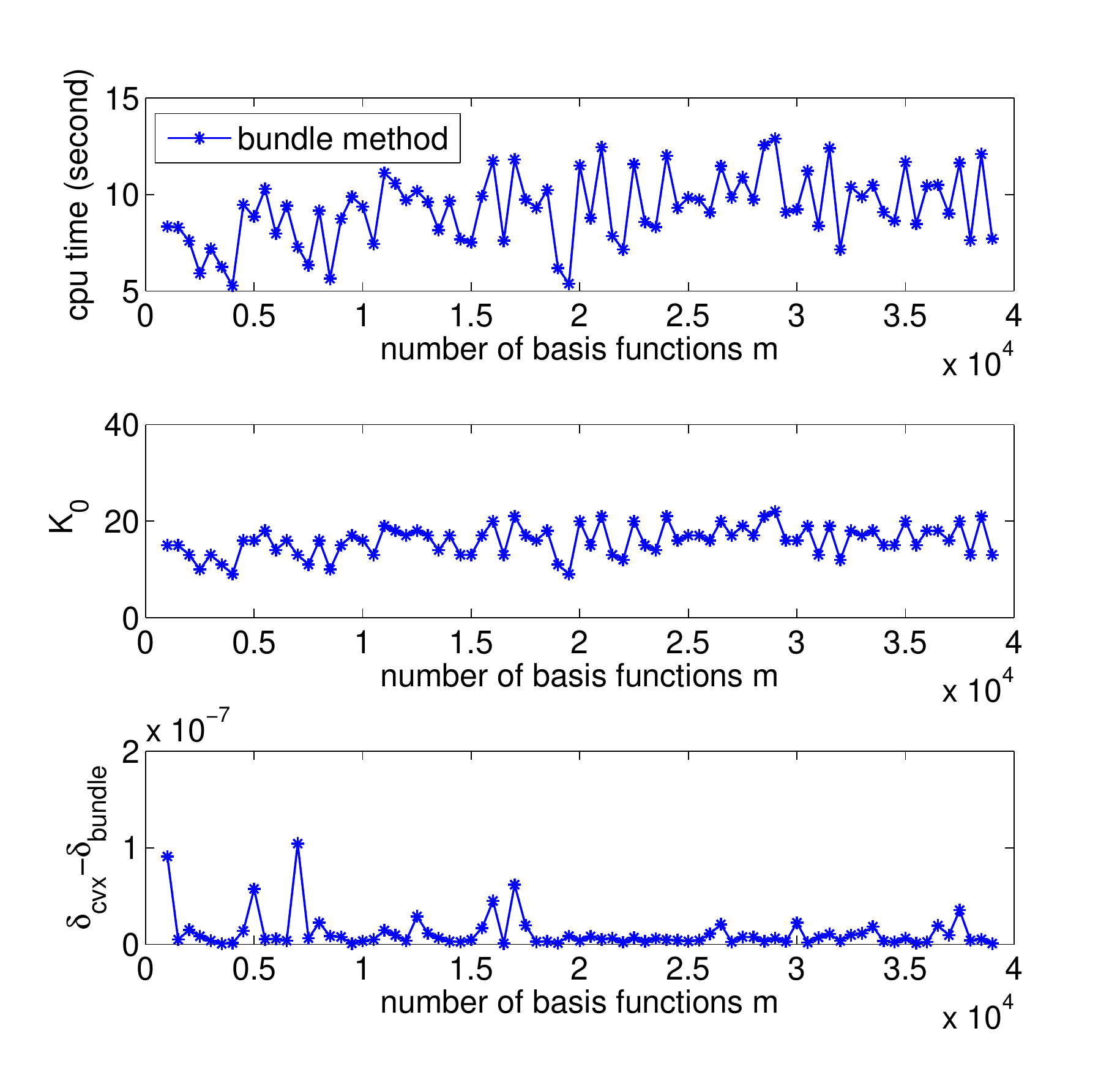}
      \caption{Computation time in seconds of the bundle method V.S. the number of basis functions $m$ (top), number of iterations $K_0$ of the bundle method V.S. the number of basis functions $m$ (middle),  optimal value obtained by cvx $\delta_{cvx}$ minus optimal value obtained by the bundle method $\delta_{bundle}$  V.S. the number of basis functions $m$ (bottom)}
      \label{figurelabel2}
   \end{figure}






\section*{ACKNOWLEDGMENT}

\bibliographystyle{alpha}
\bibliography{biblio}

\end{document}